\documentclass[12pt]{amsart}

\usepackage{hyperref}
\usepackage{amssymb}
\usepackage{amsmath}
\usepackage{amsthm}
\usepackage{enumerate}
\usepackage{graphicx}
\usepackage{mathrsfs}
\usepackage{extarrows}
\usepackage{color}
\numberwithin{equation}{section}

\usepackage{setspace}

\theoremstyle{plain}
\newtheorem{thm}{Theorem}[section]
\newtheorem{lem}[thm]{Lemma}

\newtheorem{cor}[thm]{Corollary}
\newtheorem{question}[thm]{Question}

\theoremstyle{definition}

\theoremstyle{remark}

\begin{document}
	
\title[A remark on the Hardy-Littlewood maximal operators]
	{A remark on the  Hardy-Littlewood maximal operators}
	
\author{Wu-yi Pan}

\address{Key Laboratory of Computing and Stochastic Mathematics (Ministry of Education),
		School of Mathematics and Statistics, Hunan Normal University, Changsha, Hunan 410081, P.
		R. China}
	
\email{pwyyyds@163.com}

\date{\today}
	

	
\begin{abstract}

We investigate the magnitude relation of the non-centered Hardy-Littlewood maximal operators and centered one. By using a discretization technique, we prove two facts: the first one is that the space is ultrametric if and only if  the two maximal operators are identical for all discrete measure;  the second is, the uncentred maximal operator is strictly greater than the centered one if $(M,d_g)$ is a Riemannian manifold and $\mu$ is the Riemannian volume measure.

\end{abstract}
	
\maketitle
\section{Introduction} 
Let $(X,d)$ denote a metric space and $\mu$ a (positive Borel) measure on $X$.  We use the symbol $B$ to denote an closed ball.  To facilitate research, we will assume that the measure of each ball is finite and the support of $\mu$ is nonempty. We consider the \emph{Hardy–Littlewood maximal function} centered and non-centered on $X$ and, for $f$ 
locally $\mu$-integrable,
\[
M^c_{\mu} f\stackrel{\rm def}{=} \sup_{r} \frac{1}{\mu(B(x,r))}\int_{B(x,r)} fd\mu
;\quad M  _{\mu} f\stackrel{\rm def}{=} \sup_{B: B\ni x} \frac{1}{\mu(B)}\int_B fd\mu.
\]
 We will write them simply $M^c f$, $Mf$ when no confusion can arise. 
It is easy to check that maximal functions are defined everywhere in the support of $\mu$. In order to avoid trivialities, we are required to $\int_B fd\mu/\mu(B)=0$ if $\mu(B)=0$. 

 In an ultrametric space, every point in a ball is a center, and so, all maximal functions coincide.  When is the converse true?  To be more specific, the question is that the statement:

\begin{question}\label{qe:1.1}
 Whether coincidence of the two operators $($in the support of the measure$)$ implies that the metric is ultra? If not, what is the structure of the metric inherited from the parent space on the support? Is it ultra?
\end{question}
Regretfully, a simple delta in a non-ultrametric space provides a counter-example to the first issue.
Here, we solve Question \ref{qe:1.1} partly. Indeed, we show 

\begin{thm}\label{thm:1.2}
	Let $(X,d)$ be a metric space. The metric $d$ is ultra iff the two maximal operators are the same for every discrete measure on $X$.
\end{thm}

In addition, we also show a rigid result revealing how the size relation of maximal functions changes  affect the geometric structure of the support of the measure.  
\begin{thm}\label{thm:1.3}
	If there is a continuous measure $\mu$ such that $M^c_{\mu}f=M_{\mu}f$ for all $f\in L^{1}_{\operatorname {loc}}(\mu),$ then there does not exist a series of the point $\{x_{n}\}_{n\geq 1}\subseteq \operatorname {supp}(\mu)$ so that each point in the sequence is the midpoint of two other different points in the sequence and $\{x_{n}\}_{n\geq 1}$ has a limit that is also in $\operatorname {supp}(\mu)$. 
\end{thm}

It provides a verification of intuition that the non-center maximal operator is strictly greater than centered one for a large class of common spaces. Clearly, Theorem \ref{thm:1.3} implies

\begin{cor}
	Let $(M,d_g)$ be a Riemannian manifold, where the metric $d_g$ induced by the inner product $g$. Then for each continuous measure $\mu$ satisfying that the support of the measure $\mu$ contains a geodesic curve, there exists a function $f\in L^{1}(\mu)$ such that its uncentred maximal function and centered maximal function take different values at least at one point.
\end{cor}

\section{Proofs}\label{S2}
 The complete proof of the following can be found in \cite[Lemma 3.1]{PD22}. See also \cite {Ko15}.

\begin{lem}\label{lem:2.1}
Define the action of $M^c$ on a finite Borel measure  by \begin{equation*}
	M^c_{\mu} \nu\stackrel{\rm def}{=} \sup_{r} \frac{\nu(B(x,r))}{\mu(B(x,r))},
\end{equation*}
for every $x\in \operatorname {supp}(\mu)$. $M_\mu\nu$ is similarly considered.  

Consider a sequence of finite Borel measures $\nu_n$ on $X$.
If $\nu_n$ converges
weakly to a finite Borel measure  $\nu$, then $\liminf\limits_{n \to \infty}M\nu_n(x)\geq M\nu(x)$ for all $x\in \operatorname {supp}(\mu)$. The same fact holds for $M^c$.
 
\end{lem}

For our purpose, we next apply a  discretization method to establish a crucial lemma.

\begin{lem}\label{lem:2.2}
	If there is a measure $\mu$ such that  $M^c_\mu f=M_\mu f$ for $f\in L^{1}(\mu)$, then for all pairs of points $(x, y)\in  \operatorname {supp}(\mu)$, we have
	\begin{equation}\label{eq:2.1}
\mu B(y,d(x,y)) 	\leq	\inf_{B\ni x,y} \mu B.
	\end{equation}
Furthermore, $\mu B(y,d(x,y))=\mu B(x,d(x,y))$.
\end{lem}

\begin{proof}
	If the support of the measure contains only one point, the theorem follows, so we can take two different points in  $\operatorname {supp}(\mu)$. 
	For fixed different points $x,y\in \operatorname {supp}(\mu)$, choose $\delta > 0$ and let  
	\begin{equation*}
		f_\delta=\frac{ \chi_{U(x,\delta)\setminus U(y,d(x,y))}}{\mu(U(x,\delta)\setminus U(y,d(x,y)))}
	\end{equation*}
where $\chi_A$ is the characteristic function of the set $A$ and $U(x,r)$ denotes the open ball centered at $x$ with radius $r>0$. Clearly, $f_\delta \in L^1$, and $f_\delta d\mu$ converges
weakly to $\delta_x$ where $\delta_x$ denotes the Dirac delta concentrated at  $x\in X$.
	  Applying Lemma \ref{lem:2.1} with $f_\delta$, we know 
	  \begin{equation*}
	  	\liminf\limits_{\delta \to 0}Mf_\delta(z)\geq M\delta_x(z)
	  \end{equation*}
	  for all $z\in \operatorname {supp}(\mu)$.  Now	observe that\begin{equation*}
	  	M^cf_\delta(y)=\sup_{r\in[d(x,y),d(x,y)+\delta]}\frac{\mu(B(y,r)\cap U(x,\delta)\cap U^c(y,d(x,y)))}{\mu(U(x,\delta)\cap U^c(y,d(x,y)))\mu(B(y,r))}\leq \frac{1}{\mu B(y,d(x,y))}.
	  \end{equation*} 
  Thus $\limsup\limits_{\delta \to 0}M^cf_\delta(y)\leq \frac{1}{\mu B(y,d(x,y))}$. Comparing these inequalities, we conclude that 
  \begin{equation}\label{eq:2.2}
  	M\delta_x(y)\leq1/\mu B(y,d(x,y)).
  \end{equation}
 It follows from the definition that $M\delta_x(y)=1/\inf_{B\ni x,y}\mu B$. Substituted it into the above inequality \eqref{eq:2.2}, the proof is completed.
\end{proof}

The lemma  has made the function of ``facing a kick of the door". Now we are in the position of the proofs of Theorem \ref{thm:1.2} and \ref{thm:1.3}. 

%\begin{re}
%	One can wonder if $M^cf_\delta$  will converge to $M^c\delta_y$. However, in general, it is negative if we consider the measure $\mu = m+\delta_{-1}$ by taking $x,y$ to be the real numbers $1,0$ correspondingly, where $m$ is the Lebesgue measure on $\mathbb{R}$.
%\end{re}
\begin{proof}[Proof of Theorem \ref{thm:1.2}.]
	 The adequacy is clear. Conversely, supposing that for every discrete measure on $X$, $M^cf=Mf$ for $f\in L^{1}(\mu)$, we then need to show the metric is ultra.   We prove it by contradiction. Suppose, if possible, that the metric space $(X,d)$ is not ultrametric, then there exists three points $x,y,z$ such that $z\in B(x,d(x,y))$ and $d(z,y)> d(x,y)$. This means $z\notin B(y,d(x,y)).$ Now pick a measure $\nu=\delta_x+\delta_y+\delta_z$. Applying Lemma \ref{lem:2.2} with $\nu$, we have $\nu B(y,d(x,y))=2=3=\nu B(x,d(x,y))$. It is impossible so we have proved Theorem \ref{thm:1.2}.
\end{proof}
\begin{proof}[Proof of Theorem \ref{thm:1.3}.]
	We also argue it by contradiction. Suppose, if possible, that there exists a sequence $\{x_{n}\}_{n\geq 1}$ described in Theorem \ref{thm:1.3}. By passing to a subsequence, we can suppose that the sequence satisfies that $x_{i+2}$ is the midpoint of  $x_i$ and $x_{i+1}$  for $i\geq 1$ and $\lim\limits_{n\to \infty}x_n=x_0$ for some $x_0\in \operatorname {supp}(\mu).$ First consider the closed balls $B_i=B(x_{i+2},d(x_{i+1},x_{i+2}))$ for all $i\geq 1$. Next we use \eqref{eq:2.1} to get $\mu(B_i)\leq \mu(B_{i+1})$.  It follows from the sequence  $B_i$ is monotonically decreasing that $\mu(B_i)=\mu(B_{i+1})$, 
and hence $\mu(B_1)=\lim\limits_{i\to \infty}\mu(B_i)=\mu(\{x_0\})=0$ by the continuity of the measure. It contradicts the fact that $x_0\in \operatorname {supp}(\mu)$. This finishes the proof.
\end{proof}

%\begin{lem}
%	If there is a measure such that $Mf(x)> M^uf(y)$ for some $x,y \in X$, then $Mf(x)$ is attained or $Mf(x)=\lim\limits_{r\to 0}\langle f\rangle_{B(x,r)}$. 
%\end{lem}
%\begin{proof}
%	To show it, we prove 
%	$
%		Mf(x)=\sup\limits_{r\in[0,d(x,y)]}\langle f\rangle_{B(x,r)},
%	$ then the lemma follows by passing to a subsequence.  We prove it by contradiction. Suppose, if possible, that the extreme radii diverge to infinity or is attained for some $r>d(x,y)$. If $Mf(x)<+\infty$, so for any $\varepsilon$, we can select a $r_\varepsilon>d(x,y)$ such that $\langle f\rangle_{B(x,r_\varepsilon)}\geq Mf(x)-\varepsilon$. Now  $B(x,r_\varepsilon)\ni y$ implies $Mf^u(y)\geq Mf(x)-\varepsilon$ which obviously contradicts $Mf(x)> M^uf(y)$.
%	
%	Next, we have to consider the case of $Mf(x)=+\infty$.  The same of
%	reduction to absurdity shows a contradicting to $M^uf(y)<+\infty$.
%\end{proof}
%\begin{cor}\label{co:2.2}
%	Let $x$ be an accumulation point on $X$.
%If everywhere neighborhood of $x$ contains another point $y$ satisfying $Mf(x)> M^uf(y)$, then $Mf(x)=\lim\limits_{r\to 0}\langle f\rangle_{B(x,r)}$. 
%\end{cor}

\end{document}